\documentclass[12pt]{amsart}

\title{Checkerboard graph monodromies}
\author{S.~Baader, L.~Lewark, L.~Liechti}

\usepackage{amssymb}
\usepackage{graphicx}
\usepackage{color}
\usepackage{url}
\usepackage{microtype,hyperref,extarrows,calc,bbm}
\usepackage[nameinlink,nosort]{cleveref}
\let\cref\Cref
\hypersetup{colorlinks={true},linkcolor={black},citecolor={black},filecolor={black},urlcolor={black},pdfauthor={\authors},pdftitle={\shorttitle}}
\newcommand{\myemail}[1]{\texttt{\href{mailto:#1}{#1}}}

\theoremstyle{plain}
\newtheorem{theorem}{Theorem}
\newtheorem{corollary}[theorem]{Corollary}
\newtheorem{proposition}[theorem]{Proposition}
\newtheorem{lemma}[theorem]{Lemma}
\newtheorem{problem}[theorem]{Problem}

\theoremstyle{definition}
\newtheorem*{definition}{Definition}

\theoremstyle{remark}

\newtheorem*{remark}{Remark}

\def\R{{\mathbb R}}

\begin{document}

\begin{abstract} We associate an open book with any connected plane checkerboard graph, thus providing a common extension of the classes of prime positive braid links and positive tree-like Hopf plumbings. As an application, we prove that the link type of a prime positive braid closure is determined by the linking graph associated with that braid.
\end{abstract}

\maketitle

\section{Introduction}
Braid groups play an important role at the interface between geometry, topology and algebra. While being special cases of mapping class groups, braid groups allow for a complete formulation of knot theory, via their geometric realisation in the 3-sphere \cite{Bi}. Our main objects of interest are positive braids, which form a monoid that captures the essence of braids while being relatively small. It is big in that every braid can be written as a product of a central element with a positive braid, e.g.\ by the Garside form \cite{Ga}. It is small in that closures of positive braids -- positive braid links -- share very special features with algebraic links, which they include \cite{Ei}.
One of these features is that positive braid links bound canonical genus-minimising Seifert surfaces \cite{Sta}, which happen to be fibre surfaces in the case of non-split braids. These surfaces can be constructed as unions of discs and twisted ribbons by the well-known Seifert algorithm \cite{Sei}.
They serve as a starting point for a graph theoretical model for positive braid links, which is the main topic of this article, and which we now describe in a slightly informal way.
\begin{figure}[h]
\begin{center}
\def\svgwidth{200pt}
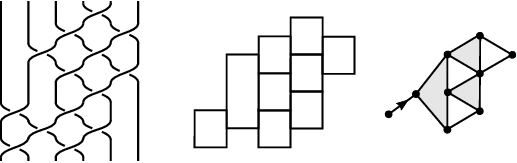
\caption{A positive braid, its brick diagram and linking graph. Regions are coloured black/white to indicate clockwise/anticlockwise orientation of the boundary cycle.}
\label{fig:1}
\end{center}
\end{figure}

A positive braid can be encoded in a plane graph with vertical lines and horizontal edges, called brick diagram, see \cref{fig:1}.
The number of bricks, i.e.\ innermost rectangles, equals the first Betti number of the Seifert surface. Since all crossings are positive, we can reconstruct a braid word from its brick diagram. The linking graph of a brick diagram is a subgraph of its dual graph, where all edges corresponding to non-linked bricks are deleted. Here a pair of adjacent bricks is non-linked if their intersection is contained in the interior of a side of one of the bricks, see again \cref{fig:1}.

Linking graphs come with a natural orientation consistent with a checkerboard colouring (see \cref{defs} and \cref{fig:1,fig:2}). Throughout this paper, the term linking graph refers to the embedded linking graph together with its orientation. The isotopy type of the linking graph does not determine the original braid word, as demonstrated by the pair of positive braids $\sigma_1^3$ and $\sigma_1 \sigma_2 \sigma_1 \sigma_2$, whose linking graph is a single edge. However, these two braids have isotopic closures: the positive trefoil knot. This is a special case of our main result.

\begin{theorem} \label{braidthm} The linking graph of a prime positive braid word~$\beta$ determines the oriented link type of the closure of~$\beta$.
\end{theorem}

It is worth noting that the two standard braid representatives of the torus link $T(p,q)$, $(\sigma_1 \sigma_2 \ldots \sigma_{p-1})^q$ and $(\sigma_1 \sigma_2 \ldots \sigma_{q-1})^p$, have very similar linking graphs: one is the mirror of the other. What is more intriguing, both graphs are realised by precisely one more braid word:
\begin{align*}(\sigma_p \sigma_{p+1} \cdots \sigma_{p+q-1})(\sigma_{p-1} \sigma_p \cdots \sigma_{p+q-2}) \cdots (\sigma_1 \sigma_2 \cdots \sigma_q),\\
(\sigma_q \sigma_{q+1} \cdots \sigma_{q+p-1})(\sigma_{q-1} \sigma_q \cdots \sigma_{q+p-2}) \cdots (\sigma_1 \sigma_2 \cdots \sigma_p).
\end{align*}
The former is depicted in \cref{fig:2} for $(p,q)=(4,5)$. 
It is not too difficult to check that these braid words indeed represent the same torus link. In fact, some linking graphs are realised by many different braid words, for example trees.

\begin{figure}[h]
\begin{center}
\def\svgwidth{200pt}
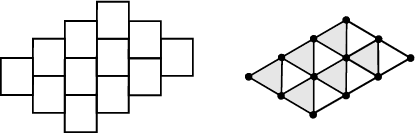
\caption{An alternative diagram of $T(4,5)$.}
\label{fig:2}
\end{center}
\end{figure}

The way of proving \cref{braidthm} is by reconstructing the monodromy of a prime positive braid link from its linking graph. With a connected linking graph, we will associate an abstract surface $\Sigma$ with boundary together with a homeomorphism $\varphi: \Sigma \to \Sigma$. Such a pair $(\Sigma,\varphi)$ is called an \emph{open book} and determines a fibred link in a 3-manifold. Our description makes crucial use of the bipartite nature of the dual graphs of linking graphs. As we will see, the construction of the open book $(\Sigma,\varphi)$ extends to all connected \emph{checkerboard graphs}, i.e.\ finite simple connected plane oriented graphs whose bounded regions carry a checkerboard colouring. The orientation of an induced boundary cycle of a black or white face is required to be in the clockwise or anticlockwise sense, respectively.
Furthermore, we need an additional technical assumption, which is discussed in \cref{defs}.
As a consequence, we obtain a natural extension of the class of prime positive braid links.

\begin{theorem} \label{graphthm} Every connected checkerboard graph $\Gamma$ determines a unique open book $(\Sigma,\varphi)$ of a strongly quasipositive fibred link $L$ in $S^3$. If $\Gamma$ is the linking graph of a positive braid word $\beta$, then $L$ is isotopic to the closure of $\beta$.
\end{theorem}
Here a strongly quasipositive fibred link is a fibred link whose corresponding open book supports the unique tight contact structure on $S^3$~\cite{He}.
The class of links associated with connected checkerboard graphs also includes the class of weight two arborescent links, i.e.\ links arising as the boundary of a plumbing of positive Hopf bands along plane trees.
Indeed, we will see that for $\Gamma$ a plane tree with a certain orientation,
the corresponding fibred link $L$ coincides with the fibred arborescent link obtained by plumbing positive Hopf bands along the tree, as described in~\cite{BS}.

The next section makes the definition of checkerboard graphs and their relation to positive braids precise.
In order to associate an open book with connected checkerboard graphs, we need to single out a conjugacy class of Coxeter elements in certain Artin groups. This is done in \cref{Cox} and is based on an extension of Steinberg's Lemma for trees to checkerboard graphs~\cite{Ste}.
The construction of open books is then carried out in \cref{checkopenbooks},
and \cref{sec:thm2} contains the core of the proof of \cref{graphthm}.
We observe that \cref{graphthm} implies \cref{braidthm}.
\Cref{sec:orientation} discusses the significance of the orientation of checkerboard graphs, in particular for trees.
In \cref{sec:table}, we exhibit an example of a checkerboard graph whose associated knot is neither the closure of a positive braid, nor a weight two arborescent knot.
For this purpose, we discuss how to recover the Seifert matrix of a knot from its checkerboard graph, and we present a list of all weight two arborescent knots and positive braid knots of genus five or less.
The paper closes with a list of problems and questions in \cref{sec:perspectives}.
\medskip

\paragraph{\emph{Acknowledgements:}} 
We thank Filip Misev for explaining us the fibre surface of a cycle,
Peter Feller for confusing us with fake cycles,
Francis Bonahon for discussing the symmetries of trees,
and Pierre Dehornoy for pointing out that linking graphs admit a checkerboard colouring.
\vspace{-2ex}

\setcounter{tocdepth}{1}
\tableofcontents
\section{Brick diagrams, linking and checkerboard graphs}
\label{defs}

A \emph{positive braid word} is a finite product of positive generators
$\sigma_1$, $\sigma_2,\ldots,\sigma_{n-1}$ of the braid group $B_n$ on $n$ strands. We may think of a positive braid word (up to commuting non-adjacent braid generators) as a brick diagram, i.e.\ a plane graph with $n$ vertical lines connected by horizontal arcs, one for each crossing. Brick diagrams are special cases of fence diagrams for strongly quasi-positive braids, as introduced by Rudolph~\cite{Ru}. They naturally embed as retracts into the canonical Seifert surface associated with the closure of positive braids. The bricks thereby correspond to embedded positive Hopf bands, whose core curves form a basis for the first homology group of the Seifert surface. We call two bricks linked if the intersection number of their core curves is non-zero. This happens precisely if they are arranged in the same pattern as the two bricks of the braids $\sigma_i^3$, $\sigma_i \sigma_{i+1} \sigma_i \sigma_{i+1}$, $\sigma_{i+1} \sigma_i \sigma_{i+1} \sigma_i$. The linking graph associated with a positive braid word is a subgraph of the plane dual graph of the brick diagram, with one vertex for each brick and one edge for each pair of linked bricks, see \cref{fig:1,fig:2,fig:3}.
This can be seen as a generalisation of Dynkin diagrams for links of plane curve singularities, compare \cite[Ex.~1]{Lo}.

As mentioned in the introduction, the plane isotopy type of the linking graph does not determine the original braid word. If the linking graph is not connected, it does not even determine the link type of the positive braid's closure. For example, the two braid words $\sigma_1 \sigma_2 \sigma_2 \sigma_1 \in B_3$, $\sigma_1 \sigma_1 \sigma_3 \sigma_3 \in B_4$ define two non-isotopic positive braid links (one of which is split) with the same linking graph: two points.
Restricting oneself to non-split positive braids does not alleviate this, since e.g.\ $\sigma_1^4 \sigma_2^2 \sigma_3^2$ and $\sigma_1^2 \sigma_2^4 \sigma_3^2$ define two non-isotopic non-split positive braid links, but have the same linking graph (a disjoint union of two isolated vertices and a path of length two).
We will therefore only deal with positive braid words whose closure is prime. These are precisely those positive braid words with a connected linking graph, since positive braids are visually prime \cite{Cro}.
By a theorem of Stallings~\cite{Sta}, this condition guarantees that the canonical Seifert surface of the braid closure is a fibre surface.

An important feature of linking graphs is that they come with a checkerboard colouring, i.e.\ a black-and-white colouring of their bounded regions, 
with alternating colours at all internal edges. Indeed, all regions of a linking graph are bounded by cycles that look like triangles with a distinguished vertex on the left or right. 
We colour these black and white, respectively. Triangles sharing an edge are of different type, compare \cref{fig:1,fig:2}.
There are three types of edges in a linking graph: vertical edges, edges with positive slope and edges with negative slope. We orient vertical edges downwards, and the other two types upwards (see \cref{fig:3}): in this way,
the boundary cycles of black and white regions are oriented in the clockwise and anticlockwise sense, respectively. Moreover, an edge corresponds to a non-trivial intersection of the
two homology classes associated with its endpoints -- and this orientation reflects the sign of that intersection. This will be more fully discussed in \cref{checkopenbooks,sec:thm2}.
\begin{figure}[t]
\begin{center}
\def\svgwidth{115pt}
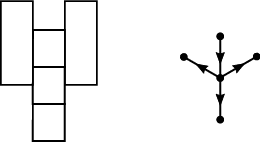
\caption{An intersection tree.}
\label{fig:3}
\end{center}
\end{figure} 

Linking graphs are the motivation for the following definition of checkerboard graphs.
\begin{definition}
A \emph{checkerboard graph} is a finite simple oriented plane graph satisfying two conditions:
\begin{itemize}
\item Every bounded region has a coherently oriented induced cycle as boundary.
\item There exists a set of edges that contains exactly one edge out of the
boundary of every bounded region, and contains at least one, but not all, edges out of any oriented cycle.
\end{itemize}
\end{definition}
Some comments on this definition are in order.
The set of edges in the second condition is not part of the data -- its existence is required,
but no choice of a specific set is made. Note that the second part of the second condition is equivalent to the requirement that reversing the orientation
of the selected edges yields an acyclic graph. The first condition leaves the orientation of bridges free;
on all other edges, the first condition means precisely that the orientation is induced by a checkerboard colouring
of the bounded regions, where the boundaries of black and white regions are oriented in the clockwise and anticlockwise sense, respectively.

A linking graph satisfies both conditions, and is thus a special case of a checkerboard graph.
For the second condition, one may select e.g.\ all edges with positive slope. Each of the triangle-shaped regions
contains exactly one of those edges, and reversing the orientation of all those edges means that every non-vertical edge is oriented towards the left -- so there can be no oriented cycles.

Another special case of a checkerboard graph is an embedded tree. It satisfies both conditions, no matter how the edges are oriented,
and so we call oriented plane trees \emph{checkerboard trees}.
Some plane trees arise as linking graphs of positive braid words (for example a star with four edges, obtained from $\sigma_2^2 \sigma_1 \sigma_3 \sigma_2^2 \sigma_1 \sigma_3$, see \cref{fig:3}), 
but not all (for example a star with five or more edges). In any case, the degree of a vertex of a linking graph cannot exceed six. 

In the next two sections, we will associate an open book with any connected checkerboard graph.

\section{Coxeter elements for checkerboard graphs}
\label{Cox}
Let $\Gamma$ be a checkerboard graph. 
We consider the \emph{right-angled Artin group} $A(\Gamma)$ defined by $\Gamma$, i.e.\ the group given by the following presentation.
There is one generator for every vertex of $\Gamma$, and two generators commute if and only if there is no edge connecting the corresponding vertices in $\Gamma$. 
There are no other relations.
We are interested in elements of $A(\Gamma)$ which are represented by words in which every generator appears exactly once (and to the power $1$). 
Such elements have been widely studied in the context of Coxeter groups, and are called \emph{Coxeter elements}.
We will call the corresponding words \emph{Coxeter words}.
The remainder of this section is devoted to the proof of the following lemma,
which associates a unique conjugacy class of Coxeter elements with each checkerboard graph.

\begin{lemma}
\label{uniqueconjugacyclass}
There is an enumeration $v_1, \ldots, v_n$ of the vertices of a given checkerboard graph $\Gamma$
that traverses the vertices on the boundary cycle of any black (or white) region of $\Gamma$ in clockwise (or anticlockwise) order.
The graph $\Gamma$ determines the Coxeter element $v_1 v_2 \cdots v_n \in A(\Gamma)$ uniquely, up to conjugation.
\end{lemma}
As we will see, this enumeration is reflected in the product order of positive Dehn twists in a positive braid monodromy.
\begin{proof}
The main tool for the proof is a one-to-one correspondence (due to Shi~\cite{Shi}) between acyclic orientations on $\Gamma$
(not to be confused with the given checkerboard orientation on $\Gamma$)
and enumerations $v_1, \ldots, v_n$ of the vertices of $\Gamma$, up to \emph{shuffling},
i.e.\ switching the indices of $v_i$ and $v_{i+1}$, if there is no edge between them.
The correspondence is defined as follows:
given an enumeration, we orient an edge towards its endpoint of higher index. Note that this
does give an acyclic orientation, and shuffling the enumeration does not change the resulting orientation.
Conversely, with an acyclic orientation, we associate the following enumeration:
start the enumeration with all sources, in any order. Remove the sources, and continue inductively with the sources of the remaining graph.  Here, we use that any acyclicly oriented finite graph has at least one source.
Note that at any stage, no two sources are connected by an edge; so enumerating them in a different order
just gives a shuffled enumeration. One easily checks that these two assignments are mutually inverse.

Now, let us prove the existence of the enumeration as claimed.
By the second condition of the definition of a checkerboard graph, we may select a set of edges
containing exactly one edge out of the boundary of every bounded region and at least one, but not all edges of every oriented cycle. 
Then, starting from the checkerboard orientation of $\Gamma$, one can reverse the orientation of every selected edge.
This yields an acyclic orientation, and we claim that the associated enumeration of vertices has the desired property.
To see this, consider a region of $\Gamma$.
Its boundary is an induced cycle $C$. The acyclic orientation disagrees with the checkerboard orientation of $\Gamma$ for exactly one of the edges in $C$, so that $C$ has one sink, and one source. The source must come first in the enumeration of vertices; then all the other vertices must follow in clockwise or anticlockwise order, for a black or white region, respectively. This is precisely the condition
imposed on the enumeration.
\begin{figure}[t]
\def\svgwidth{110pt}
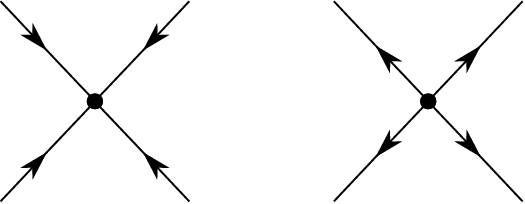
 \caption{Pushing down a maximal vertex in an acyclic orientation yields another acyclic orientation.}
 \label{pushdown}
\end{figure}

Finally, let us show that all such enumerations define the same Coxeter element $v_1 v_2 \cdots v_n \in A(\Gamma)$, up to conjugation.
Note that the associated acyclic orientation $\mathfrak{o}$ gives the enumeration back, up to shuffling. Moreover, shuffled enumerations
define the same Coxeter element. In this way, we can associate Coxeter elements with acyclic orientations.
Suppose we are given another enumeration satisfying the condition of the lemma, and denote its associated acyclic orientation by $\mathfrak{p}$.
We now have to show that the Coxeter elements associated with $\mathfrak{o}$ and $\mathfrak{p}$ are conjugate.

For this purpose, we consider a local move transforming one acyclic orientation of $\Gamma$ into another,
called \emph{pushing down a maximal vertex}. By this, we mean reversing all the edge orientations around a maximal vertex,
i.e.\ a sink, see \cref{pushdown}. Suppose $\mathfrak{o}'$ is obtained from $\mathfrak{o}$ by pushing down $v$.
Since $v$ is a sink, we may assume after shuffling that the enumeration $v_1, \ldots, v_n$ associated with $\mathfrak{o}$ satisfies $v = v_n$.
Then $v_n, v_1, \ldots, v_{n-1}$ is an enumeration associated with $\mathfrak{o}'$. So the Coxeter elements associated
with $\mathfrak{o}$ and $\mathfrak{o}'$ are conjugate.

Thus it remains to prove that $\mathfrak{o}$ and $\mathfrak{p}$ are related by a sequence of pushing down maximal vertices.
For that, it is sufficient that $\mathfrak{o}$ and $\mathfrak{p}$ have the same flow difference, as proven by Pretzel~\cite{Pretzel}.
Here, the \emph{flow difference} of an orientation assigns to each oriented cycle in $\Gamma$
the number of edges traversed in the positive sense (with respect to the orientation) 
minus the number of edges traversed in the negative sense (with respect to the orientation). 
Note that due to its linearity properties, the flow difference is determined by its values on a cycle basis of $\Gamma$. 
A natural cycle basis is given by the boundary cycles of the bounded regions, and the condition the enumerations satisfy
immediately implies that they have the same flow differences:
each clockwise or anticlockwise boundary cycle around a black or white region, respectively, has flow
difference equal to its length minus 2.
\end{proof}
\section{Checkerboard open books}
\label{checkopenbooks}
The goal of this section is to associate an open book with each connected checkerboard graph.
Furthermore, if we start with the linking graph of a positive braid, we wish to obtain the open book associated with the fibre structure of the complement of its closure.

\subsection{Constructing the surface and the twist curves}
Let $\Gamma$ be a connected checkerboard graph. For each vertex $v_i$ of $\Gamma$, let $A_i$ be an oriented annulus and let $\gamma_i\subset A_i$ be a core curve of $A_i$, oriented in the anticlockwise sense. 
We construct the surface $\Sigma$ by suitably gluing the annuli~$A_i$.

For every vertex $v$ of $\Gamma$, the planar embedding of $\Gamma$ determines a circular ordering of the incident edges.
We glue two annuli $A_i$ and $A_j$ along a rectangle $R_k$ if their corresponding vertices $v_i$ and $v_j$, respectively, are connected by an edge $e_k$ in~$\Gamma$. 
The rectangle $R_k$ is taken so that its four edges alternatingly belong to the boundary of $A_i$ and~$A_j$.
Furthermore, we want our gluing to respect the circular orderings induced by the planar embedding of~$\Gamma$. 
More precisely, let $v_i$ be a vertex of degree $\ell$ and let $e_{k_1}e_{k_2}\cdots e_{k_{\ell}}$ be the circular ordering of the incident edges.
Then, the circular ordering of the gluing rectangles on $A_i$ should read $R_{k_1}R_{k_2}\cdots R_{k_{\ell}}$, see \cref{circularordering} for an example.
\begin{figure}[t]
\def\svgwidth{240pt}
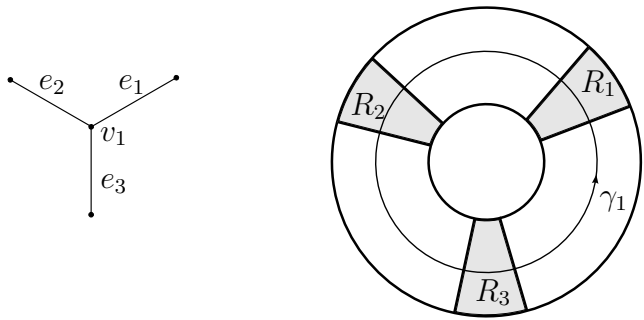
\caption{The embedding into the plane determines a circular ordering of edges around a vertex, which is the circular order in which the annuli are glued.}
\label{circularordering}
\end{figure}
\begin{figure}[t]
\def\svgwidth{190pt}
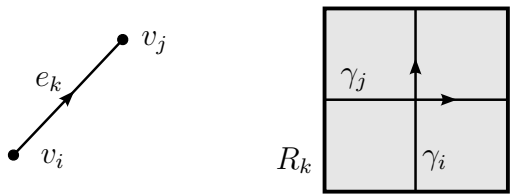
\caption{The orientation of an edge determines in which of the two possible ways two annuli are glued.}
\label{positiveintersection}
\end{figure}

A priori, there are two possibilities for two annuli $A_i$ and $A_j$ to be glued together along a rectangle~$R_k$. Either the intersection of the corresponding core curves 
$\gamma_i$ and $\gamma_j$ is positively or negatively oriented.
In order to determine how we glue, we make use of the orientation of~$\Gamma$. 
If $e_k$ is an oriented edge starting at $v_i$ and ending at $v_j$, we choose the core curve $\gamma_i$ to intersect the core curve $\gamma_j$ positively,
as shown in \cref{positiveintersection}.

\begin{figure}[t]
\def\svgwidth{250pt}
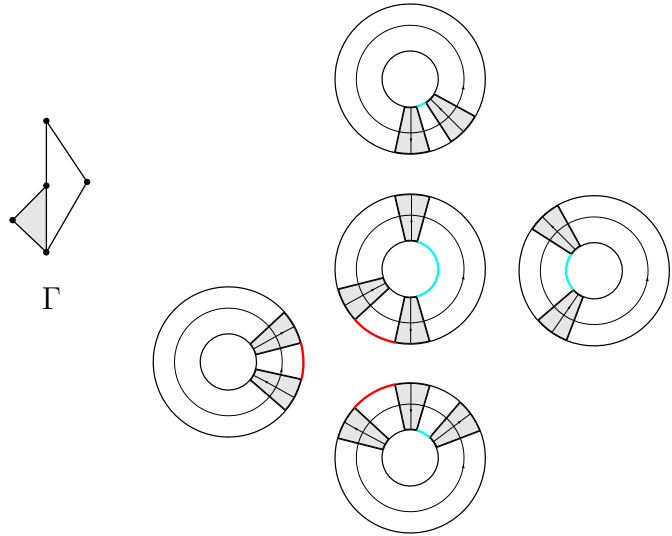
\caption{A checkerboard graph $\Gamma$ and the corresponding abstract surface $\Sigma$.
Discs are glued in along the red (fine dashed) and the blue (coarse dashed) boundary components.}
\label{surfaceexample}
\end{figure}
As a last step, for each coloured region of $\Gamma$, we glue one disc along the boundary of the surface we obtained so far.
More precisely, for a white region, we glue a disc along the inner boundary of the annuli corresponding to the vertices on the boundary of the region, 
for a black region, we glue a disc along the outer boundary. This is shown for an example in \cref{surfaceexample}.

\subsection{Choosing the twist order} \label{subsec:twistorder}
So far we have constructed a surface $\Sigma$ and a collection of simple closed curves $\gamma_i$ from a connected checkerboard graph~$\Gamma$. 
In order to define an open book, we also have to specify a mapping class (up to conjugation) on~$\Sigma$. 
We want to define the mapping class as a product of positive Dehn twists along the curves $\gamma_i$ such that every curve gets twisted along exactly once.
What we have to do is choose a product order, i.e.\ enumerate the vertices of~$\Gamma$.
Note that the two twists along $\gamma_i, \gamma_j$ commute if the corresponding vertices are not connected by an edge.
Therefore, the subgroup of the mapping class group of $\Sigma$ generated by Dehn twists along the $\gamma_i$ is a quotient
of the Artin group $A(\Gamma)$.
Applying \cref{uniqueconjugacyclass} now gives a mapping class of $\Sigma$, uniquely defined up to conjugation.
We call this conjugacy class the \emph{checkerboard monodromy} associated with the connected checkerboard graph $\Gamma$.
Recalling \cref{uniqueconjugacyclass}, one sees that the checkerboard monodromy comes from a product order $v_1 \cdots v_n$
in which the vertices on the boundary of a bounded region of $\Gamma$ occur in cyclic order -- clockwise for a black region,
and anticlockwise for a white region.

\begin{remark}
Our construction of an open book can be generalised in several ways.
The orientation of a checkerboard graph simultaneously determines the
intersection of the core curves of the annuli and the twist order of the monodromy.
But of course one can specify these two parameters in a different fashion, and independently from each other.
More generally, one can allow different types of graphs or even negative
Dehn twists, cf.\ constructions of Hironaka~\cite{Hi0, Hi}.
Our focus is to capture the features of open books associated with
positive braids.
\end{remark}

\section{Proof of \texorpdfstring{\cref{braidthm} and \cref{graphthm}}{Theorem 1 and Theorem 2}}
\label{sec:thm2}
This section is devoted to the proof of \cref{graphthm}. 
We first deal with the case where $\Gamma$ is the linking graph of a positive braid word. 
We show that the checkerboard monodromy defined in \cref{checkopenbooks} 
equals the monodromy of the fibred link defined by the closure of the positive braid. 

Let $\Sigma_\beta$ be the canonical Seifert surface associated with a prime positive braid word~$\beta$. 
By a theorem of Stallings, the closure of $\beta$ is a fibred link and $\Sigma_\beta$ is a fibre surface~\cite{Sta}.  
Furthermore, the monodromy is a product of positive Dehn twists along certain curves $\alpha_i$ corresponding to bricks in the brick diagram. 
The canonical Seifert surface and the twist curves are depicted for the braid $\beta= \sigma_2\sigma_1\sigma_3\sigma_2\sigma_1\sigma_2\sigma_3\sigma_2$ in \cref{braidsurfaceexample}.
\begin{figure}[b]
\def\svgwidth{160pt}
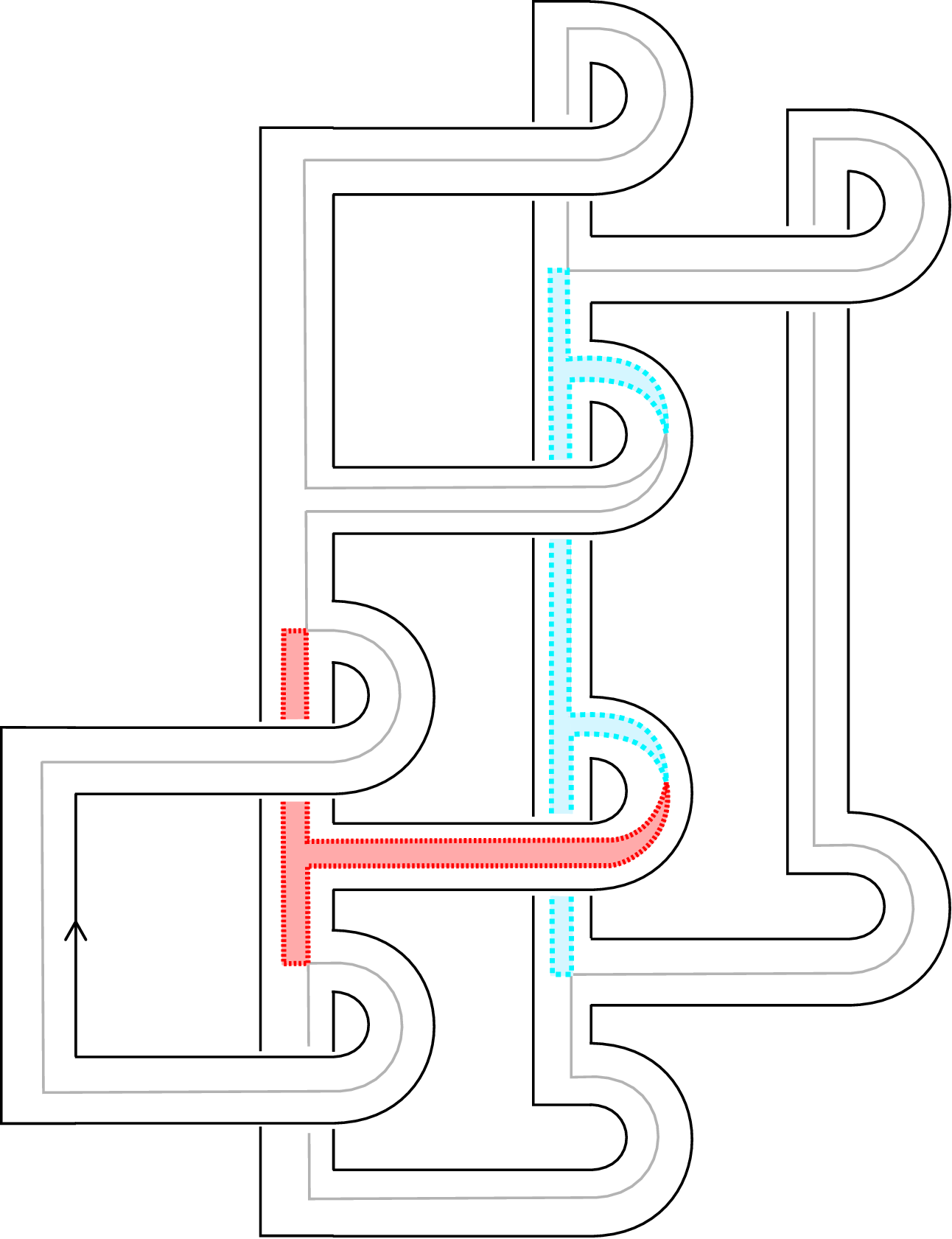
\caption{The canonical Seifert surface of the positive braid $\sigma_2\sigma_1\sigma_3\sigma_2\sigma_1\sigma_2\sigma_3\sigma_2$ 
is the fibre surface. The monodromy is a product of positive Dehn twists along the grey curves.}
\label{braidsurfaceexample}
\end{figure}
More precisely, the canonical Seifert surface is obtained by successive positive Hopf plumbing via adding hook-like handles from bottom to top within every column, 
proceeding from the rightmost column to the leftmost. 
Every added handle after the first within a column describes a positive Hopf plumbing, 
with the core curve passing through the added handle and the one just below, compare with \cref{braidsurfaceexample}.
This leads to a Hopf band $H_i$ with core curve $\alpha_i$ per brick of the brick diagram of~$\beta$. 
By definition, two core curves $\alpha_i$ and $\alpha_j$ intersect if and only if their corresponding bricks are linked.
In particular, there is a one-to-one correspondence between vertices of the linking graph $\Gamma$ and Hopf bands $H_i$ 
with core curves $\alpha_i$ in the plumbing construction of the canonical Seifert surface of~$\beta$.

The monodromy of a single positive Hopf band is a positive Dehn twist along its core curve. Therefore, the monodromy of a positive braid is a product of positive Dehn twists
along the core curves of the Hopf bands in the plumbing construction, since the monodromy of a plumbing is the product of the monodromies of the plumbing summands~\cite{Gab,Sta}.
We remark that in the cyclic order of this Dehn twist product, the twists corresponding to a boundary cycle around a black or white region appear in a clockwise or anticlockwise order, respectively.

Let $\Sigma_\Gamma$ and $\gamma_i$ be the surface and the twist curves, respectively, obtained by the construction described in \cref{checkopenbooks} applied to the linking graph~$\Gamma$ of $\beta$.
By \cref{uniqueconjugacyclass}, we are done if we can show that abstractly, the surface $\Sigma_\beta$ and the core curves $\alpha_i$ agree 
with the surface and the twist curves obtained by the construction discussed in \cref{checkopenbooks}
applied to the connected checkerboard graph $\Gamma$ associated with~$\beta$. 
Here we orient the core curves $\alpha_i$ so that they run anticlockwise when drawn as in \cref{braidsurfaceexample}. 
\newline

\noindent
\emph{Claim. There exists a homeomorphism from $\Sigma_\beta$ to $\Sigma_\Gamma$, sending the oriented core curves $\alpha_i$ to the oriented twist curves $\gamma_i$.}
\newline

Define a bijection $f$ from the intersection points of the core curves $\alpha_i$ to the intersection points of the twist curves~$\gamma_i$. 
Such a bijection exists, since by construction, both sets of intersection points are in a natural bijection with the set of edges of~$\Gamma$.
This bijection $f$ can be extended to a homeomorphism $F:\bigcup\alpha_i \to \bigcup\gamma_i$, respecting the orientations of the curves $\alpha_i$ and~$\gamma_i$. 
This extension is possible since the cyclic order in which a core curve $\alpha_i$ 
intersects other core curves $\alpha_j$ equals the cyclic order of the edges incident to the vertex corresponding to~$\alpha_i$.
By construction, the same holds for the curves~$\gamma_i$.

Note that the unions $\bigcup\alpha_i$ and $\bigcup\gamma_i$ fill the surfaces $\Sigma_\beta$ and $\Sigma_\Gamma$, respectively, 
i.e.\ the complement consists of discs and boundary-parallel annuli.
To prove the claim, it therefore suffices to show that a boundary cycle of $\Sigma_\beta\setminus\bigcup\alpha_i$ bounds a disc if and only if its image under $F$ bounds a disc in $\Sigma_\Gamma$.
This is indeed the case, since both $\bigcup\alpha_i$ and $\bigcup\gamma_i$ bound exactly one disc for every bounded region of the complement of the checkerboard graph $\Gamma$.
Furthermore, in both cases, the disc lies to the right or to the left of the curves depending on whether the region is coloured black or white, respectively, compare with \cref{surfaceexample,braidsurfaceexample}. 
This proves the claim and hence the part of \cref{graphthm} concerning positive braids and their linking graphs.

We still need to prove that open books associated with arbitrary connected checkerboard graphs define strongly quasipositive fibred links in $S^3$. For this, we use that a positive stabilisation of an open book does not change the type of the resulting contact 3-manifold (see e.g.~\cite[Section~2]{E}), and preserves strong quasipositivity \cite{He,Ru2}. We claim that the open book associated with a connected checkerboard graph with $n$ vertices is obtained from the trivial open book $(D^2,\text{Id})$ by $n$ times iterated positive stabilisations. Indeed, every connected checkerboard graph $\Gamma \subset \R^2$ contains a vertex $v$ adjacent to the unbounded region such that $\Gamma \setminus v$ is still a connected checkerboard graph.
By construction, the surface $\Sigma_\Gamma$ is obtained from $\Sigma_{\Gamma \setminus v}$ by adding several 1-handles and 2-handles. By cancellation, one may equivalently add a single 1-handle. The monodromy of $\Sigma_{\Gamma}$ is given (up to conjugation) by composing the monodromy of $\Sigma_{\Gamma \setminus v}$ with a Dehn twist along a curve that runs once through that 1-handle. This is precisely a positive stabilisation. Therefore, we are done by induction on the number of vertices.  

\section{Orientation, invertibility and mutants}
\label{sec:orientation}
In this section, we discuss the effect a change of orientation or of the embedding of a checkerboard graph has on the associated link.
Let us first focus on bridges.
Throughout this section, let $\Gamma$ be a connected checkerboard graph with a bridge $e$ and let $\Gamma_1, \Gamma_2$ be the connected components of $\Gamma\setminus e$, which are also connected checkerboard graphs.
\begin{lemma}\label{lem:plumb}
The fibre surface $\Sigma_{\Gamma}$ is a plumbing of $\Sigma_{\Gamma_1}$ and $\Sigma_{\Gamma_2}$.
\end{lemma}
\begin{proof}
The edge $e$ determines a square on each of $\Sigma_{\Gamma_1}$ and $\Sigma_{\Gamma_2}$, as indicated in \cref{circularordering}.
Those squares lie on the annuli corresponding to the endpoints of $e$.
Note that there are two ways to plumb $\Sigma_{\Gamma_1}, \Sigma_{\Gamma_2}$ in an orientation-preserving way along those squares.
Those two ways are distinguished by the sign of the intersection
of the core curves of the two annuli corresponding to the endpoints of $e$. Consider the plumbing surface for which that sign conforms with the orientation of the edge $e$ in the sense of \cref{positiveintersection}.
This plumbing surface is a fibre surface in $S^3$ whose monodromy is the product of the monodromies of $\Sigma_{\Gamma_1}, \Sigma_{\Gamma_2}$~\cite{Sta}.
It is canonically homeomorphic to $\Sigma_{\Gamma}$,
and the order in which Dehn twists appear in the product of the monodromies respects the condition of \cref{subsec:twistorder}.
Thus the product of monodromies equals the checkerboard monodromy.
This concludes the proof, since the monodromy uniquely determines the isotopy type of the fibre surface.
\end{proof}

\begin{proposition}\label{prop:mut}
Reversing the orientation of $e$ yields a checkerboard graph $\Gamma'$
whose associated link $L'$ is a positive mutant of $L$.
\end{proposition}
\begin{proof}
By \cref{lem:plumb}, the fibre surfaces $\Sigma_{\Gamma}, \Sigma_{\Gamma'}$
arise as the two different possible plumbings of $\Sigma_{\Gamma_1}, \Sigma_{\Gamma_2}$ along the squares given by $e$.
There is a ball in $S^3$ whose intersection with $\Sigma_{\Gamma}$ is one of those plumbed fibre surfaces.
The boundary of this ball intersects $\Sigma_{\Gamma}$ in the plumbing square, and $L$ in four points;
cutting the ball out, rotating it by 180$^{\circ}$ and regluing it yields $\Sigma_{\Gamma'}$.
So $L$ and $L'$ are mutants, and  since the mutation was performed without reversing the orientation of the tangle contained in the ball, they are positive mutants.
\end{proof}
Next, let us consider $-\Gamma$, the checkerboard graph obtained from $\Gamma$ by reversing the orientation of \emph{all} edges.
Denote by $-L$ the \emph{inverse} of $L$, i.e.\ the link obtained by reversing the orientation of all components of $L$.
\begin{proposition}\label{prop:reverse}
The link associated with $-\Gamma$ is $-L$.
\end{proposition}
\begin{proof}
Let $\varphi$ be the checkerboard monodromy of $\Sigma_{\Gamma}$. Reversing the orientation of $\Sigma_{\Gamma}$
yields a Seifert surface $-\Sigma_{\Gamma}$ for $-L$, which is a fibre surface with monodromy $\varphi^{-1}$.
Indeed, one easily sees that the mapping torus of $(-\Sigma_{\Gamma}, \varphi^{-1})$ is homeomorphic to the mapping torus of $(\Sigma_{\Gamma}, \varphi)$.
Thus it suffices to check that the open book associated with $-\Gamma$ is equivalent to $(-\Sigma_{\Gamma}, \varphi^{-1})$.
To this end, we first observe that the surface associated with $-\Gamma$ differs from $\Sigma_{\Gamma}$ merely by the intersections of core curves of annuli --
all of them are of opposite sign.
So the surface associated with $-\Gamma$ is indeed canonically homeomorphic to $-\Sigma_{\Gamma}$. 
Next, we note that the monodromy $\varphi$ is defined as a product $T_1T_2\cdots T_n$ of Dehn twists in a certain order.
From the construction of checkerboard open books it is evident that taking the reverse order yields the monodromy of the surface associated with $-\Gamma$.
Moreover, if we identify that surface with $-\Sigma_{\Gamma}$, each individual twist is the inverse of the corresponding twist on $\Sigma_{\Gamma}$,
because all Dehn twists on $\Sigma_{\Gamma}$ as well as on $-\Sigma_{\Gamma}$ are positive. Hence the monodromy associated with $-\Gamma$ is $T_n^{-1} \cdots T_2^{-1} T_1^{-1} = \varphi^{-1}$. This concludes the proof.
\end{proof}

\begin{figure}[t]
\newcommand{\myinclude}[1]{\raisebox{-.45\height}{\includegraphics[scale=.64]{figs/#1.pdf}}}
\Crefname{proposition}{Prop.}{Prop.}
\Crefname{corollary}{Cor.}{Cor.}
\parbox{1.6cm}{\centering\myinclude{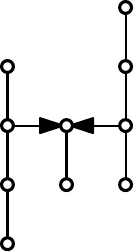}\\[1ex]\footnotesize $13$n$_{241}$}
$\xrightarrow{\text{\cref{prop:reverse}}}$
\parbox{1.6cm}{\centering\myinclude{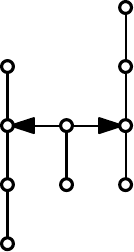}\\[1ex]\footnotesize $-13$n$_{241}$}
$\xlongequal{\text{\cref{cor:move}}}$
\parbox{1.6cm}{\centering\myinclude{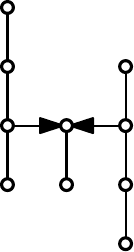}\\[1ex]\footnotesize $-13$n$_{241}$}
$\xlongequal{\text{reflection}}$
\parbox{1.6cm}{\centering\myinclude{reverse_ex1}\\[1ex]\footnotesize $13$n$_{241}$}\\[9ex]

\parbox{1.6cm}{\footnotesize\ $13$n$_{300}$\\[1.5ex]\centering\myinclude{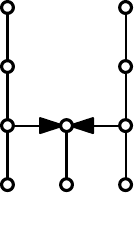}}
$\xrightarrow{\text{\cref{prop:reverse}}}$
\parbox{1.6cm}{\footnotesize\ $-13$n$_{300}$\\[1.5ex]\centering\myinclude{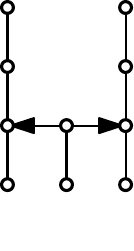}}
$\xlongequal{\text{\cref{cor:move}}}$
\parbox{1.6cm}{\footnotesize\ $-13$n$_{300}$\\[1.5ex]\centering\myinclude{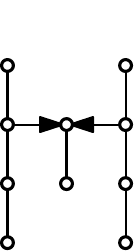}}
\makebox[\widthof{$\xlongequal{\text{reflection}}$}][c]{$\neq$}
\parbox{1.6cm}{\footnotesize\ $13$n$_{300}$\\[1.5ex]\centering\myinclude{reverse_ex4}}
\caption{The first row shows the invertibility of the weight two arborescent knot $13n_{241}$, using
its checkerboard tree. In the second row, an attempt to show the invertibility of $13n_{300}$ (a mutant of $13n_{241}$)
in a similar fashion fails -- and indeed, $13n_{300}$ is not invertible.}
\label{fig:moves}
\end{figure}

The two previous propositions immediately yield the following.
\begin{corollary}
A weight two arborescent link is a positive mutant of its inverse.\qed
\end{corollary}\noindent
Note that every knot is a mutant of its inverse, but not necessarily a positive one.
To wit, positive mutants have S-equivalent Seifert forms \cite{Kirk}, whereas general mutants need only have algebraically concordant Seifert forms \cite{Kim};
so a knot such as the $P(7,3,19)$ pretzel knot, whose Seifert form is not S-equivalent to its transpose \cite{trotter}, is not a positive mutant of its inverse.

Let us now come to the relevance of the embedding of a checkerboard graph $\Gamma$.
As a first observation, we claim that the open books associated with $\Gamma$ and the mirror image of $\Gamma$ in the plane are equivalent.
Indeed, reflecting $\Gamma$ has the effect of reversing for all vertices the cyclic orders of incident edges.
The same effect is achieved by reversing the orientations of all core curves, which does not change the open book.

Reflecting only part of $\Gamma$, however, may change the open book. Take a bridge $e$, and denote as before the connected components of $\Gamma\setminus e$ as $\Gamma_1, \Gamma_2$.
One may attach $\Gamma_1$ along $e$ to a mirror image of $\Gamma_2$, forming a connected checkerboard graph $\Gamma''$.
The surface $\Sigma_{\Gamma_2}$ is unchanged by reflecting $\Gamma_2$, but $\Sigma_{\Gamma_1}$ and $\Sigma_{\Gamma_2}$ are plumbed differently in $\Gamma''$.
So the link associated with $\Gamma''$ equals the link associated with $\Gamma'$ (the graph obtained from $\Gamma$ by reversing the orientation of $e$).

\begin{corollary}\label{cor:move}
Reversing the orientation of $e$, and taking the mirror image of one of the connected components of $\Gamma\setminus e$
yields a connected checkerboard graph $\widetilde{\Gamma}$. The links associated with $\Gamma$ and $\widetilde{\Gamma}$ are isotopic.\qed
\end{corollary}
In particular, if $\Gamma_1$ or $\Gamma_2$ is symmetric under reflection, the orientation of $e$ does not matter for
the link type of $L$. In those cases, we omit the orientation of $e$ from drawings of $\Gamma$.
\cref{fig:moves} shows example applications of \Cref{prop:reverse} and \cref{cor:move}.

\begin{corollary}\label{cor:halfreverse}
Let $\Gamma$ be a checkerboard tree. Let $V$ be a subset of the vertices, containing exactly one endpoint of each edge.
Reversing the cyclic orderings of edges around all vertices in $V$ yields a checkerboard tree with associated link $-L$.
\end{corollary}
\begin{proof}
Apply \cref{cor:move} successively to every edge $e$ of $\Gamma$, obtaining a graph $\widetilde{\Gamma}$ with associated link $L$.
At every step, half of $\Gamma$ is reflected, and the
cyclic ordering of edges around vertices in that half changes. If this happens an even number of times for some vertex --
so the cyclic ordering around that vertex has not changed in total -- then it must have happened an odd number of times
for all adjacent vertices. So the cyclic ordering is reversed for precisely the vertices in $V$ (or precisely for those
not in $V$, in which case we replace $\widetilde{\Gamma}$ by its mirror image). Finally, reversing the orientations of
all edges of $\widetilde{\Gamma}$ yields the graph given in the statement, and so the
associated link is $-L$ by \cref{prop:reverse}.
\end{proof}

Aside from \cref{prop:mut}, there are other moves on checkerboard graphs that yield mutants.
Namely, let $\Gamma$ be a connected checkerboard graph with a cut-vertex $v$. Let $\Gamma_1, \ldots, \Gamma_n$ be the connected components of $\Gamma\setminus v$,
and $\overline{\Gamma_i}$ the subgraph of $\Gamma$ induced by $\Gamma_i \cup \{v\}$. One may glue together the $\overline{\Gamma_i}$ along their copies of the vertex $v$,
in any cyclic order around $v$. In this way one obtains a family of $(n-1)!$ connected checkerboard graphs, whose associated links are positive mutants.
Up to reflection, at most $(n-1)!/2$ of these graphs are different, so this will yield non-isotopic mutants only for vertices of degree 4 or higher.
One may also reverse the orientations of the $\Gamma_i$, which will result in mutants which are generally not positive.

Finally, let us compare the oriented links associated with checkerboard trees with 
Bonahon-Siebenmann's unoriented arborescent links~\cite{BS}, which are associated with plane unoriented weighted trees.
First off, since the weights give the self-linking of the plumbed bands, we will set all weights equal to $+2$, so all bands are positive Hopf bands.
With this restriction on the weights, Bonahon and Siebenmann associate links with unoriented plane trees $\Theta$. The bands are plumbed in such a way that the core curve of any fixed band intersects
the core curves of all adjacent bands with the same sign. We choose an orientation of $\Theta$ that makes every vertex a sink or a source, and
denote the tree with this orientation by $\Gamma$.
There are precisely two such orientations for $\Theta$.
Our construction associates with $\Gamma$ an oriented link -- and forgetting its orientation
gives the unoriented link the Bonahon-Siebenmann construction associates with $\Theta$. Conversely, given an oriented plane tree $\Gamma$, one may simultaneously change its
embedding and its orientation using \cref{cor:move}, without changing the associated link. By a repeated application of this move, one obtains an oriented plane tree~$\Gamma'$,
all whose vertices are sinks and sources. Forgetting the orientation gives an unoriented tree $\Theta$ with which Bonahon and Siebenmann associate a link of the same unoriented type
as the link we associate with $\Gamma$.

\begin{figure}
\newcommand{\myinclude}[1]{\raisebox{-.45\height}{\includegraphics[scale=.43]{figs/#1.pdf}}}
\parbox{6em}{\raggedleft
$13n_{241}$:
{\footnotesize
$\sigma_1^3\sigma_2^2\sigma_1^2\sigma_2\sigma_3\sigma_2^3\sigma_3$}}
 \myinclude{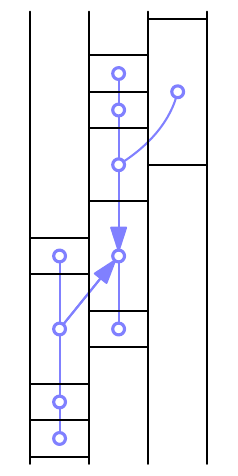}
\hfill
\parbox{8em}{\raggedleft
$-13n_{300}$:
{\footnotesize
$\sigma_1^3\sigma_2^2\sigma_1^2\sigma_3\sigma_2\sigma_4\sigma_3^3\sigma_4$}}
\myinclude{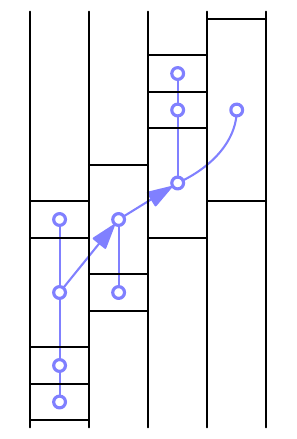}\hspace*{1em}\mbox{}
\caption{Two positive braid words with the same unoriented checkerboard graph, and closures of different unoriented knot type.}
\label{fig:unoriented}
\end{figure}

Hence our and Bonahon-Siebenmann's construction give the same set of (unoriented) links. Still, the orientation of edges is crucial
for the purpose of recovering positive braid links from their linking graph.
Indeed, there exist pairs of positive braid words (see e.g.~\cref{fig:unoriented}) whose linking graphs are trees that are isotopic as unoriented plane graphs --
but not as oriented plane graphs. The associated links must then be mutant, but they may be of different unoriented link types. 
Consequently, one cannot recover the unoriented link type of a prime positive braid link from its linking graph deprived of its orientation.

\section{Positive braids of small genus}
\label{sec:table}

This section is devoted to an example of a connected checkerboard graph whose associated link is
neither a positive braid link, nor a weight two arborescent link.
A linking graph has maximal degree at most 6; so it is easy enough
to come up with checkerboard graphs $\Gamma$ that are neither linking graphs nor trees.
But the link associated with such a $\Gamma$ could nevertheless be isotopic to one
associated with a linking graph or a checkerboard tree.

More concretely, let us consider the checkerboard graph shown in \cref{fig:nobraidnotree}.
Splitting along the bridge that is drawn thick, one can see that the associated knot $K$ is a plumbing of the trefoil
and the knot~$13n_{5016}$.
In general, it is difficult to show that a given fibred positive knot such as $K$
is not the closure of a positive braid.
Moreover, our construction does not directly give a knot diagram of $K$.
Fortunately, it is possible to compute the Seifert matrix and hence the Alexander polynomial of $K$ from $\Gamma$, see \cref{prop:seifert} below.
There are only finitely many
positive braid knots of any fixed genus $g$, and one can list them (see below for details).
It turns out that no positive braid knot of genus $6$ has the same Alexander polynomial as $K$ -- and so $K$ is not the closure of a positive braid.

\begin{figure}[t]
\centering
\includegraphics[scale=.75]{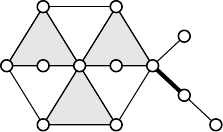}
\caption{A connected checkerboard graph whose associated knot
is neither weight two arborescent, nor a positive braid knot.}
\label{fig:nobraidnotree}
\end{figure}
\begin{proposition}\label{prop:seifert}
Let $\Gamma$ be a connected checkerboard graph. Fix an enumeration $v_1, \ldots, v_n$ of its vertices as in \cref{uniqueconjugacyclass}.
The corresponding core curves $\gamma_1, \ldots, \gamma_n$ of Hopf bands
give a basis of the first homology group of the associated fibre surface $\Sigma$ (cf.~\cref{checkopenbooks}).
With respect to this basis, one finds the following matrices:
\begin{itemize}
\item The intersection form of $\Sigma$ has the antisymmetric Gram matrix $B$ given by $B_{ij} = 1$ if there is an edge $v_i \to v_j$, and $B_{ij} = 0$ if there is no edge between $v_i$ and $v_j$.
\item The Dehn twist along $\gamma_k$ acts on $H_1(\Sigma)$ by the matrix $S^k$ with $S^k_{ij}$ equal to $1$ if $i = j$, equal to $B_{ij}$ for $i = k \neq j$, and equal to $0$ otherwise.
\item The monodromy acts on $H_1(\Sigma)$ by the matrix $S = S^n\cdot\ldots\cdot S^1$.%
\item If $(S - \mathbbm{1})$ is invertible (e.g.\ for a knot), then $\Sigma \subset S^3$
has Seifert matrix $A = B\cdot (S - \mathbbm{1})^{-1}$.
\end{itemize}
\end{proposition}
\begin{proof}
The formulae for the matrix of the intersection form and of the monodromy are evident from the construction of $\Sigma_{\Gamma}$.
For the matrices of Dehn twists, note that a twist along $\gamma_i$ changes precisely the homology classes of those $\gamma_j$ that intersect $\gamma_i$ non-trivially.
The formula for the matrix of the Seifert form is taken from \cite{Le}; it follows from the well-known relationships $B = A^{\top} - A$ and $S =  A^{-1}A^{\top}$.
\end{proof}
\begin{figure}[p]
\newcommand{\myinclude}[1]{\raisebox{-.45\height}{\includegraphics[scale=.32]{figs/#1.pdf}}}
\newcommand{\smallspace}{\rule{0.1em}{0pt}}
{
$g=1,2,3:$\qquad\makebox{\parbox{3em}{\raggedleft$3_1$ {\footnotesize$T(2,3)$\smallspace}}\myinclude{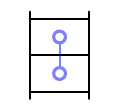}}
\makebox{\parbox{3em}{\raggedleft$5_1$ {\footnotesize$T(2,5)$\smallspace}}\myinclude{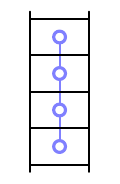}}
\makebox{\parbox{3em}{\raggedleft$7_{1}$ {\footnotesize$T(2,7)$\smallspace}}\myinclude{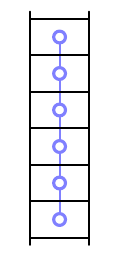}}
\makebox{\parbox{3em}{\raggedleft$8_{19}$ {\footnotesize$T(3,4)$\smallspace}}\myinclude{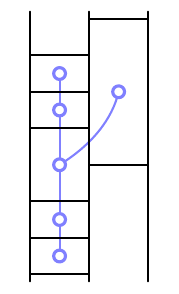}}\\[.45\baselineskip]
$g=4:$\qquad\makebox{\parbox{3.2em}{\raggedleft$9_{1}$ {\footnotesize$T(2,9)$\smallspace}}\myinclude{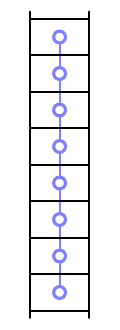}}
\makebox{\parbox{3em}{\raggedleft$10_{124}$ {\footnotesize$T(3,5)$\smallspace}}\myinclude{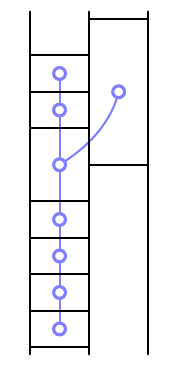}}
\makebox{$10_{139}$\myinclude{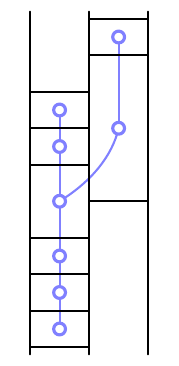}}
\makebox{$10_{152}$\myinclude{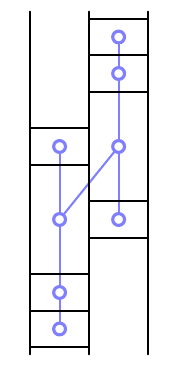}}
\makebox{$11n_{77}$\myinclude{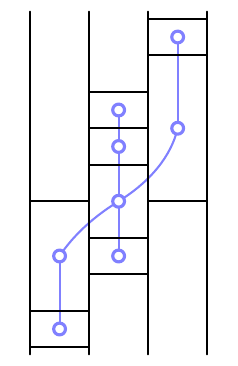}}\\[.45\baselineskip]
$g=5:$\qquad
\makebox{\parbox{3.6em}{\raggedleft$11a_{367}$ {\footnotesize$T(2,11)$\smallspace}}\myinclude{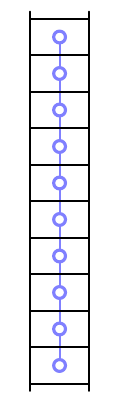}}\quad 
\parbox{3.5em}{\myinclude{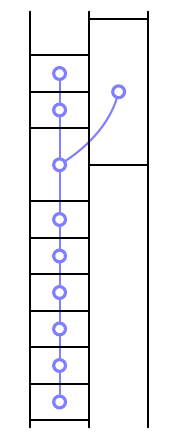}\\$12n_{242}$} \medskip
\parbox{3.5em}{\myinclude{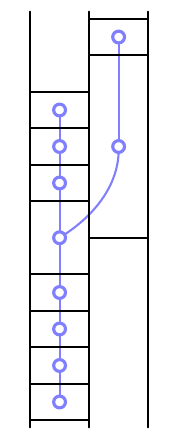}\\$12n_{472}$}
\parbox{3.5em}{\myinclude{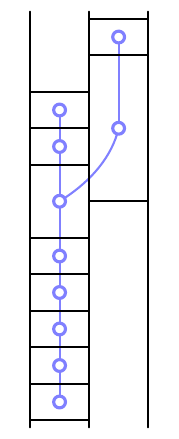}\\$12n_{574}$}
\fbox{%
\parbox{3.5em}{\myinclude{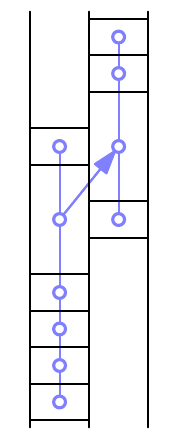}\\$12n_{679}$}
\parbox{3.5em}{\myinclude{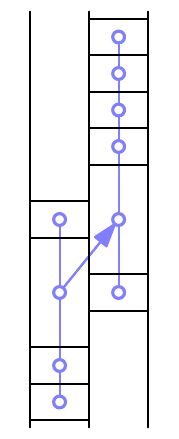}\\$-12n_{679}$}}
\fbox{%
\parbox{3.5em}{\myinclude{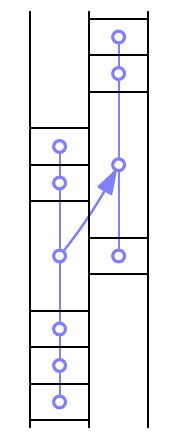}\\$12n_{688}$}
\parbox{3.5em}{\myinclude{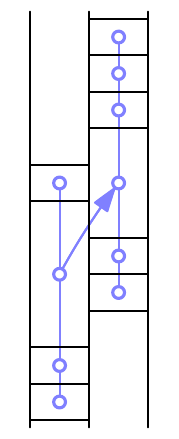}\\$-12n_{688}$}}
\parbox{3.5em}{\myinclude{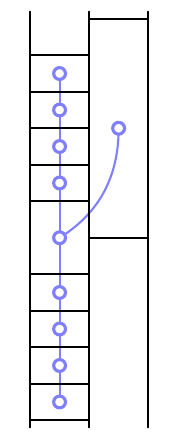}\\$12n_{725}$}
\parbox{3.5em}{\myinclude{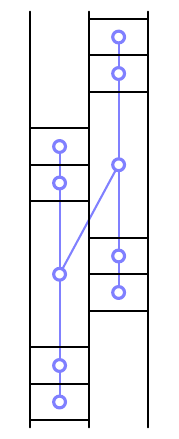}\\$12n_{888}$} \medskip
\fbox{%
\parbox{3.5em}{\myinclude{13n0241}\\$13n_{241}$}
\parbox{3.5em}{\myinclude{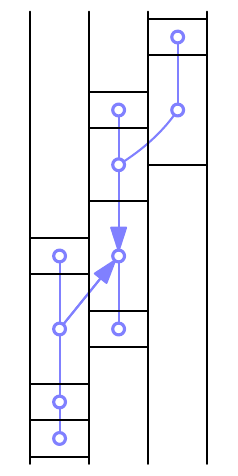}\\$13n_{300}$}
\parbox{3.5em}{\myinclude{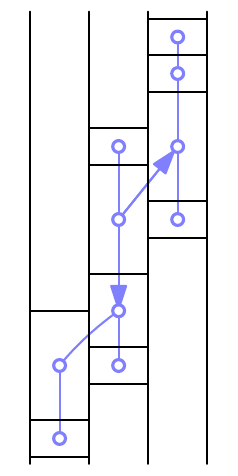}\\$-13n_{300}$}}
\parbox{3.5em}{\myinclude{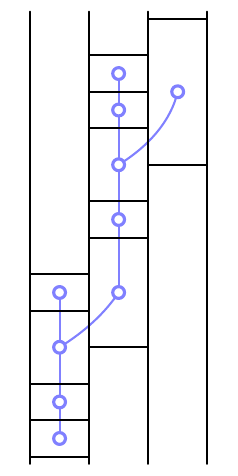}\\$13n_{604}$}
\fbox{%
\parbox{3.5em}{\myinclude{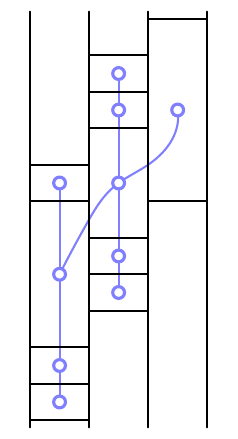}\\$13n_{981}$}
\parbox{3.5em}{\myinclude{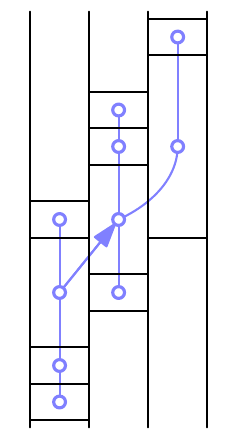}\\$13n_{1104}$}
\parbox{3.7em}{\myinclude{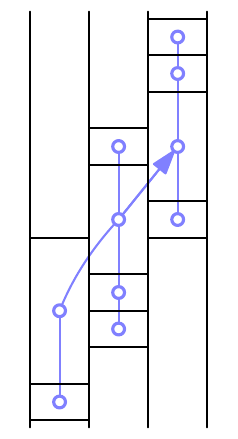}\\$-13n_{1104}$}}
\parbox{3.5em}{\myinclude{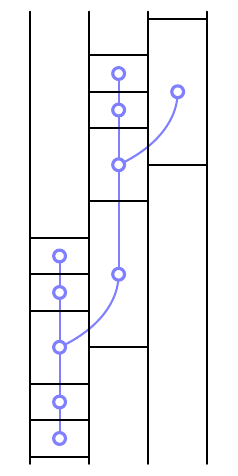}\\$13n_{1176}$} \medskip
\fbox{%
\parbox{3.5em}{\myinclude{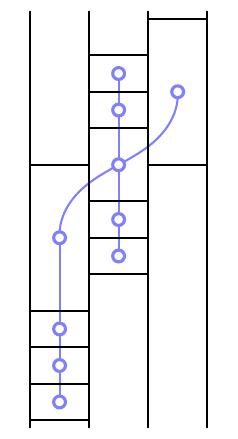}\\$13n_{1291}$}
\parbox{3.1em}{\myinclude{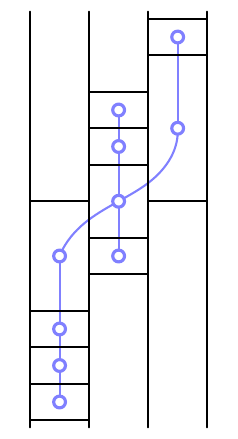}\\$13n_{1320}$}}
\parbox{4.5em}{\myinclude{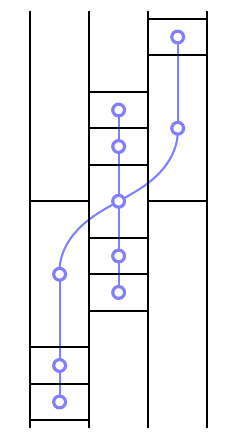}\\$13n_{2405}$}
\parbox{4.5em}{\myinclude{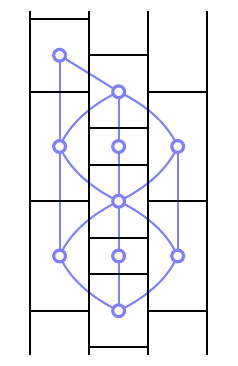}\\$13n_{4587}$}
\parbox{4.5em}{\myinclude{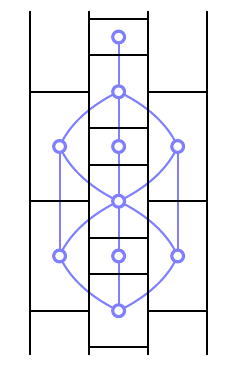}\\$13n_{5016}$}
\parbox{4.5em}{\myinclude{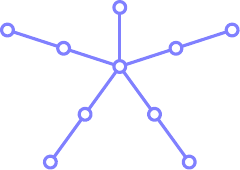}\\[1ex]$14n_{5644}$}
}\\[\baselineskip]
\caption{\small All positive braid knots and all weight two arborescent knots with genus five or less, including all such knots with twelve or fewer crossings.
Only the edge orientations which matter for the knot type are drawn.
All knots on this page except the last three are both positive braids and weight two arborescent.
Boxes contain groups of mutant knots.}
\label{fig:table}
\end{figure}

To exclude that the knot is associated with a checkerboard tree, one could use a similar brute-force argument.
However, we have a more conceptual obstruction at our disposal: the signature of the knot is 8,
and we claim that a weight two arborescent knot of genus 6 has signature 10 or 12.
Indeed, let $\Gamma'$ be a checkerboard tree. Let $(v_1,v_2)$ be an edge of $\Gamma'$ and denote the components of $\Gamma'\setminus (v_1,v_2)$ by $\Gamma'_1, \Gamma'_2$ with $v_i \in \Gamma'_i$.
Suppose neither $\Gamma'_1$ nor $\Gamma'_2$ has signature defect, i.e.\ the symmetrised Seifert form restricted to the subspace
of homology generated by $\Gamma'_i$ is positive definite. Then $\Gamma'$ has signature at least 10.
The case remains that at least one of $\Gamma'_1, \Gamma'_2$ has defect, w.l.o.g.\ take~$\Gamma'_1$.
Trees with five or less vertices have no defect, and there are precisely two trees with six vertices and defect (cf.~\cite{Ba}).
So if $\Gamma'_2$ has defect, too, then $\Gamma'_1$ and $\Gamma'_2$ must both be equal to one of two possible trees with six vertices;
but those trees correspond to links with three and five components, respectively, and their plumbing cannot yield a knot.
So $\Gamma'_2$ cannot have defect. Now assume that $(v_1, v_2)$ has been chosen such that $\Gamma'_1$ has the minimal number of vertices.
Then $\Gamma'_1 \setminus v_1$ has no defect, and so $\Gamma' \setminus v_1$ does not have defect, either. Thus $\Gamma\setminus v$ has signature at least $10$.

\Cref{fig:table} shows all positive braid knots and all weight two arborescent knots (for that class, such a list is much easier to compile than for braids)
with genus up to 5.
Let us give some details on how it was obtained. A positive braid knot $K$ of genus $g$
can be written as the closure of a positive braid word on $n$ strands for some $n$, with the property that
all generators $\sigma_1, \ldots, \sigma_{n-1}$ appear at least twice.
It follows that $n \leq 2g$, which means that there is only a finite number of braid words one has to consider
to find all braid knots of a fixed genus. Using a computer, one may iterate through all these words, and use
knotscape to identify the knots. To save time, one may restrict oneself to a smaller set of braid words, e.g.\ using
conjugation, one may suppose that every word begins with $\sigma_1$.
We compiled a list of prime positive braid knots of genus 6 or less, which will be made available on the second author's homepage.
Such a list has also been compiled by Stoimenow. The two lists are in complete agreement.\footnote{Personal communication, June 22, 2017.}

\section{Perspectives}
\label{sec:perspectives}

The classification of arborescent links allows for a complete understanding of when two plane trees define the same link~\cite{BS} (see also~\cite{Ge}). Namely, two unoriented trees give the same positive fibred link if and only if they are related by the move described in \cref{cor:halfreverse}.
We are looking for a refined move in the oriented setting
to distinguish oriented isotopy classes of checkerboard tree links.
\begin{problem} Are any two checkerboard trees with isotopic associated oriented links related by the reverse-and-reflect move of \cref{cor:move}?
\end{problem} \noindent
One may check on the basis of \cref{fig:table} that this does indeed hold for checkerboard trees whose associated links are knots of genus $5$ or less.
If it did hold in general, as a consequence one could decide the invertibility of a weight two arborescent link, and more generally find and distinguish its weight two arborescent mutants,
just by considering the combinatorics of its checkerboard tree. Stoimenow \cite{Stoi}, on the other hand, has found mutant positive braid knots,
such that the mutation cannot be seen directly from a positive  braid representation. However, these knots (16n$_{93564}$ and 16n$_{179454}$) are also weight two arborescent,
and the mutation is indeed visible from their checkerboard trees. This is a case in point that some properties of positive braid knots are more easily visible from their linking graphs
than from their braid representations.

The situation is more complicated for connected checkerboard graphs which are not trees --
 there are pairs of checkerboard graphs that are different, even disregarding orientation and embedding, with isotopic links.
One reason for that is braid conjugation, which has a somewhat mysterious effect on linking graphs, but does not change the link type. For example, the linking graph associated with the two conjugate braids $\sigma_1 \sigma_2 \sigma_1^n \sigma_2 \sigma_1$ and $\sigma_1^2 \sigma_2 \sigma_1^n \sigma_2$ is a cycle and a tree of Dynkin type $D_{n+2}$, respectively. On the other hand, a Markov move has no effect on the linking graph of a positive braid.

\begin{problem} Find a complete set of moves relating connected checkerboard graphs with equivalent links.
\end{problem}

The description of checkerboard graph links being not very explicit, it is hard to localise them in knot tables. Nevertheless, one can say that many of them admit diagrams with positive crossings only, e.g.\ the ones associated with plane trees and positive braid links. We do not know to what extent the classes of checkerboard graph links and positive fibred links coincide.

\begin{problem} Construct natural diagrams for checkerboard graph links. Are they all positive?
\end{problem}

Positive braid links and weight two arborescent links have a variety of common features. Some of these are likely to extend to checkerboard graph links, e.g.\ positivity of the signature invariant. In view of Boileau, Boyer and Gordon's work on L-space knots, it is interesting to classify checkerboard graph links with positive definite Seifert form~\cite{BB}. For positive braid links, this boils down to the classification of simply laced Dynkin diagrams~\cite{Ba}. In the case of positive braid knots, the maximality of the signature invariant is equivalent to the maximality of the topological 4-genus~\cite{Li}.

\begin{problem} Classify checkerboard graph links with maximal signature invariant ($\sigma=2g$) and maximal topological 4-genus ($g_4=g$). Do we recover the simply laced Dynkin diagrams, i.e.\ trees of type $A,D,E$?
\end{problem}

Baker's recent result on concordance of positive braid links extends to checkerboard graph links, since their fibre surfaces are plumbings of positive Hopf bands~\cite{Bk}. In particular, the existence of two non-isotopic, concordant checkerboard graph links would give a counterexample to the Slice-Ribbon conjecture. Our last problem is therefore mildly  provocative.

\begin{problem} \label{p:sr} Find a pair of smoothly concordant, non-isotopic checkerboard graph links.
\end{problem}

\medskip
\noindent
Universit\"at Bern, Sidlerstrasse 5, CH-3012 Bern, Switzerland

\medskip
\noindent
\myemail{sebastian.baader@math.unibe.ch},
\myemail{lukas.lewark@math.unibe.ch},

\noindent
\myemail{livio.liechti@math.unibe.ch}


\begin{thebibliography}{99}

\bibitem{Ba}
     S.\ Baader: \emph{Positive braids of maximal signature}, Enseign.\ Math.~\textbf{59} (2013), no.~3--4, 351--358.

\bibitem{Bk}
     K.\ L.\ Baker: \emph{A note on the concordance of fibered knots}, J.~Topol.~\textbf{9} (2016), no.~1, 1--4. 

\bibitem{Bi}
     J.\ S.\ Birman: \emph{Braids, links, and mapping class groups}, Annals of Mathematics Studies, no.\ 82. Princeton University Press, 1974.

\bibitem{BB}
     M.\ Boileau, S.~Boyer, C.~McA.~Gordon: \emph{Branched covers of quasipositive links and L-spaces}, arXiv:1710.07658.

\bibitem{BS}
     F.\ Bonahon, L.\ C.\ Siebenmann: \emph{New Geometric Splittings of Classical Knots and the Classification and Symmetries of Arborescent Knots}, available at \url{http://www-bcf.usc.edu/~fbonahon/Research/Publications.html}

\bibitem{Cro}
     P. Cromwell: \emph{Positive braids are visually prime}, Proc.~London Math.~Soc.~(3)~\textbf{67} (1993), no.~2, 384--424.

\bibitem{Ei}
     D.\ Eisenbud, W.\ Neumann: \emph{Three-dimensional link theory and invariants of plane curve singularities}, Annals of Mathematics Studies, no.\ 110. Princeton University Press, 1985.

\bibitem{E}
     J.\ Etnyre: \emph{Lectures on open book decompositions and contact structures}, Floer homology, gauge theory, and low-dimensional topology, 103--141, 
Clay Math.\ Proc., 5, Amer.\ Math.\ Soc., Providence, RI, 2006. 

\bibitem{Gab}
     D.\ Gabai: \emph{The Murasugi sum is a natural geometric operation}, Contemp.\ Math. \textbf{20} (1983), 131--143.

\bibitem{Ga}
     F.\ A.\ Garside: \emph{The braid group and other groups}, Quart.\ J.\ Math.\ Oxford Ser.~(2)~\textbf{20} (1969), 235--254. 

\bibitem{Ge}
     Y.\ Gerber: \emph{Positive Tree-like Mapping Classes}, doctoral dissertation, available at \url{http://edoc.unibas.ch/491/}

\bibitem{He}
     M.\ Hedden: \emph{Notions of positivity and the Ozsv\'{a}th-Szab\'{o} concordance invariant}, J.\ Knot Theory Ramifications \textbf{19} (2010), no.~5, 617--629. 

\bibitem{Hi0} E.\ Hironaka: \emph{Chord diagrams and Coxeter links}, J.\ London Math.\ Soc.~(2)~\textbf{69} (2004), no.~1, 243--257.

\bibitem{Hi}
     E.\ Hironaka: \emph{Mapping classes associated to mixed-sign Coxeter graphs}, arXiv:1110.1013.

\bibitem{Kim}
     S.\ Kim, C.\ Livingston: \emph{Knot mutation: 4-genus of knots and algebraic concordance}, Pacific J.\ Math.~\textbf{220} (2005), no.~1, 87--105.

\bibitem{Kirk}
     P.\ Kirk, C.\ Livingston: \emph{Concordance and mutation}, Geom.\ Topol.~\textbf{5} (2001), 831--883. 

\bibitem{Le}
     J.\ Levine: \emph{Invariants of knot cobordism}, Invent.~Math.~\textbf{8} (1969), 98--110; addendum, ibid.~\textbf{8} 1969 355.

\bibitem{Li}
     L.~Liechti: \emph{Positive braid knots of maximal topological 4-genus}, Math.\ Proc.\ Cambridge Philos.\ Soc.~\textbf{161} (2016), no.~3, 559--568.

\bibitem{Lo}
     M.\ L\"onne: \emph{Fundamental group of discriminant complements of Brieskorn-Pham polynomials}, C.~R.~Math.~Acad.~Sci.~Paris~\textbf{345} (2007), no.~2, 93--96.

\bibitem{Pretzel}
     O.\ Pretzel: \emph{On reorienting graphs by pushing down maximal vertices}, Order~\textbf{3} (1986), no.~2, 135--153.
     
\bibitem{Ru}
     L.~Rudolph: \emph{Quasipositive annuli. (Constructions of quasipositive knots and links. IV)}, J.~Knot Theory Ramifications~\textbf{1} (1992), no.~4, 451--466. 

\bibitem{Ru2}
     L.~Rudolph: \emph{Quasipositive plumbing (constructions of quasipositive knots and links. V)},
     Proc.\ Amer.\ Math.\ Soc.~\textbf{126} (1998), no. 1, 257--267. 

\bibitem{Sei}
     H.\ Seifert: \emph{\"Uber das Geschlecht von Knoten}, Math.\ Ann.\ \textbf{110} (1935), no. 1, 571--592. 

\bibitem{Shi}
     J.-Y.\ Shi: \emph{The enumeration of Coxeter elements}, J.\ Algebraic Combin.~\textbf{6} (1997), no.~2, 161--171.

\bibitem{Sta}
     J.\ R.\ Stallings: \emph{Constructions of fibred knots and links}, Algebraic and geometric topology, Proc.\ Sympos.\ Pure Math.~XXXII (1978), Part~2, 55--60, Amer.\ Math.\ Soc., Providence, R.I.

\bibitem{Ste}
     R.\ Steinberg: \emph{Finite reflection groups}, Trans.\ Amer.\ Math.\ Soc.~\textbf{91} (1959), 493--504.

\bibitem{Stoi}
     A.\ Stoimenow: \emph{On the crossing number of positive knots and braids and braid index criteria of Jones and Morton-Williams-Franks},
     Trans.\ Amer.\ Math.\ Soc.~\textbf{354} (2002), no. 10, 3927--3954. 

\bibitem{trotter}
     H.\ F.\ Trotter: \emph{On S-equivalence of Seifert matrices}, Invent.\ Math.~\textbf{20} (1973), 173--207. 

\end{thebibliography}
\end{document}